\documentclass[12pt]{article}
\usepackage{amsmath,amssymb,amsthm,amsfonts,bbm}
\usepackage{hyperref}
\usepackage[numeric]{amsrefs}

\ifx\pdftexversion\undefined
\else
\fi
\newtheorem{thm}{Theorem}[section]
\newtheorem{cor}[thm]{Corollary}
\newtheorem{lem}[thm]{Lemma}
\newtheorem{prop}[thm]{Proposition}

\theoremstyle{definition}
\newtheorem{defn}[thm]{Definition}
\newtheorem{rem}[thm]{Remark}
\newcommand{\thmref}[1]{Theorem~\ref{#1}}

\newcommand{\lemref}[1]{Lemma~\ref{#1}}
\newcommand{\secref}[1]{Section~\ref{#1}}

\numberwithin{equation}{section}
         
\newcommand\g{\gamma}
\renewcommand\d{\delta}
\newcommand\e{\varepsilon}
\renewcommand\l{\lambda}

\newcommand\G{\Gamma}
\newcommand\f{\frac}

\renewcommand{\div}{\mid}

\newcommand{\Z}{{\mathbb{Z}}}
\newcommand{\R}{{\mathbb{R}}}

\newcommand{\C}{{\mathbb{C}}}
\newcommand{\A}{{\mathcal{A}}}
\newcommand{\Q}{{\mathbb{Q}}}

\renewcommand\Re{\text{Re~}}

\newcommand{\Cl}{\operatorname{Cl}}
\newcommand{\End}{\operatorname{End}}
\newcommand{\Ell}{\operatorname{Ell}}

\newcommand{\triv}{\operatorname{triv}}
\newcommand{\iso}{\cong}
\renewcommand\O{{\mathcal O}}

\newcommand\F{{\mathbb F}}

\renewcommand\i{^{-1}}
\renewcommand\({\left(}
\renewcommand\){\right)}

\newcommand{\ignore}[1]{}
\newcommand{\myignore}[1]{}
\newcommand{\mymyignore}[1]{}

\begin{document}

\title{Expander graphs based on GRH with an application to elliptic curve cryptography}
\date{January 14, 2008}
\author{David Jao\thanks{Partially supported by NSERC Discovery Grant \#341769-07}, Stephen D. Miller\thanks{Partially supported by NSF grant DMS-0601009 and an Alfred P.
Sloan Foundation Fellowship}, and Ramarathnam Venkatesan}

\maketitle

\begin{abstract}
We present a construction of expander graphs obtained from Cayley graphs
of narrow ray class groups, whose eigenvalue bounds follow from the
Generalized Riemann Hypothesis.  Our result implies that the Cayley
graph of $(\Z/q\Z)^*$ with respect to small prime generators is an
expander.  As another application, we show that the graph of small prime
degree isogenies between ordinary elliptic curves achieves
non-negligible eigenvalue separation, and  explain the relationship
between the expansion properties of these graphs and the security of the
elliptic curve discrete logarithm problem.
\end{abstract}

\section{Introduction}

Expander graphs are widely studied in many areas of mathematics and
theoretical computer science, and such graphs are useful primarily because random walks along their edges
quickly become uniformly distributed over their vertices. Several
beautiful constructions of expanders have been based on deep tools
from representation theory and arithmetic, for example Kazhdan's
Property (T) \cite{mar} and the Ramanujan conjectures
\cite{mar2,LPS}.

The main contribution of this paper is a new, conditional construction
of expanders based on the Generalized Riemann Hypothesis (GRH), which
arises naturally in the study of the elliptic curve discrete logarithm
problem. This cryptographic connection is investigated in our parallel
paper \cite{jmv}, where it is used to establish that the discrete
logarithm problem has roughly uniform difficulty for equal sized curves. The present
paper contains a generalization of the main theorem in that paper,
along with explanations and applications of a more mathematical nature.

We briefly review some notions from graph theory, including that of
\emph{expander graph} from above. By an \emph{undirected graph}
$\G=(\mathcal V,\mathcal E)$ we mean a set of vertices $\mathcal V$
and (unoriented) edges $\mathcal E$ connecting specified pairs of
vertices.  Suppose that the graph is finite and is furthermore
\emph{$k$-regular}, meaning that there are exactly $k$ edges incident
to each vertex. The \emph{adjacency operator} $A$ acts on functions on
$\mathcal V$ by averaging them over neighbors:
\begin{equation}\label{adjc}
    (Af)(x) \ \ =  \ \ \sum_{\text{$x$ and $y$ connected by an
    edge}} f(y)\,.
\end{equation}
Since the graph is regular, the constant function
$\mathbbm{1}(x)=1$ is an eigenfunction of $A$ with eigenvalue $k$,
which is accordingly termed the \emph{trivial eigenvalue}
$\l_{\triv}$ of $A$. It is straightforward to see that the
multiplicity of $\l_{\triv}$ is equal to the number of connected
components of the graph, and that $\l_{\triv}$  is the largest
eigenvalue of $A$ in absolute value. An \emph{expander} graph is a
graph for which the nontrivial eigenvalues satisfy the bound
\begin{equation}\label{expandcondition}
      \l  \ \ \le \ \ \l_{\triv}\,(1  -  \d) \ \ \ \, \ \text{for some
      fixed
      constant $\ \d \,>\,0$\,.}
\end{equation}  If the nontrivial eigenvalues further satisfy the stronger
bound
\begin{equation}\label{expandcondition2}
      |\l|  \ \ \le \ \ \l_{\triv}\,(1  -  \d)\, ,
\end{equation} then a standard lemma (e.g.~\lemref{mixingtime}) shows that random walks of length
$\f{1}{\d}\log 2|\mathcal V|$ are equidistributed in the sense that
they land in arbitrary subsets of $\mathcal V$ with probability at
least proportional to their size. This rapid mixing of the random
walk is at the heart of most, if not nearly all, applications of
expanders.

A group $G$ generated by a subset $S=S\i$ can be made into the vertices of a
{\em
Cayley graph} $Cay(G,S)$  by defining edges from  $g$ to $sg$, for
each $s\in S$ and $g\in G$.\footnote{Note that all graphs in this paper are undirected.  We also allow for multiple edges by letting $S$ be a multiset when necessary, such as in the statement of \thmref{expthm}.} For finite abelian groups, the
eigenfunctions of $A$ are precisely the characters
$\chi:G\rightarrow \C^*$; indeed, the formula
\begin{equation}\label{abelcayeign}
    (A\chi)(x) \ \ = \ \ \sum_{s\in S} \chi(sx) \ \ = \ \ \l_\chi\,\chi(x) \
    , \ \ \text{where~} \l_\chi \ = \ \sum_{s\in S}\chi(s)\, ,
\end{equation}
shows that the spectrum consists of character sums ranging over the
generating set. The trivial eigenvalue $\l_{\triv}=|S|$ of course
comes from the trivial character $\chi={\mathbbm 1}$, and inequality
(\ref{expandcondition2}) is satisfied if the character sums for
$\l_\chi$, $\chi\not\equiv  {\mathbbm 1}$, have enough cancellation.
Abelian Cayley graphs are a restricted yet important type of graph,
and their expansion properties have been well studied
(e.g.~\cite{alonreich,lubweiss}). To be expanders, they cannot have bounded
degree but must have at least $\Omega(\log|G|)$ generators.

The expander graphs produced by our construction are abelian Cayley
graphs, and we give eigenvalue bounds for their character sums $\l_\chi$
using GRH.  Before stating the construction, we briefly recall some
terminology.  For any integral ideal $\mathfrak{m}$  in a number field $K$,
let $I_{\mathfrak{m}}$ denote the group of fractional ideals relatively
prime to $\mathfrak{m}$ (i.e.~those whose factorization into prime ideals
contains no divisor of $\mathfrak{m}$).  Let $P_{\mathfrak{m}}$ denote the
principal ideals generated by an element $k\in K^*$ such that $k\equiv
1\!\!\pmod{\mathfrak{m}}$, and let $P_{\mathfrak{m}}^+\subset P_{\mathfrak{m}}$
denote those generated by such an element $k$ which is furthermore
totally positive (i.e.~positive in all embeddings $K\hookrightarrow
\R$). The quotients $I_{\mathfrak{m}}/P_{\mathfrak{m}}$ and $I_{\frak
  m}/P_{\mathfrak{m}}^+$ are called, respectively, the ray and narrow ray
class groups of $K$ relative to $\mathfrak{m}$.

\begin{thm}\label{expthm}(``GRH Graphs'').
Let $K$ be a number field of degree $n$, $\mathfrak{m}$ an integral
ideal, and $G$ the narrow ray class group of $K$ relative to $\frak
m$. Let $q=D\cdot N\mathfrak{m}$, where $D$ is the discriminant of $K$
and $N\mathfrak{m}$ denotes the norm of $\mathfrak{m}$. Consider the set
$\{$prime ideals ${\frak p}$ coprime to $\mathfrak{m} \, | \, N{\frak
p}\le x$ is prime$\}$, and let $S_x$ denote the multiset consisting of its
image and inverse in $G$ (i.e., including multiplicities).  Then
assuming GRH for the characters of $G$, the graph $\G_x=Cay(G,S_x)$
has
\begin{equation}\label{expthmtriv}
    \l_{\triv} \ \ = \ \ 2\, \operatorname{li}(x) \ + \
    O(n\,x^{1/2}\log(xq))\,   , \ \ \ \ \operatorname{li}(x)\,=\,\int_2^x\f{dt}{\log t}\,,
\end{equation}
while the nontrivial eigenvalues $\l$ obey the bound
\begin{equation}\label{expthmnontriv}
|\l| \ \ = \ \  O(n\,x^{1/2}\log(xq))\,.
\end{equation}
In particular, if $B>2$ and $x\geq (\log q)^{B}$,
\begin{equation}\label{expthmbd}
    |\l| \ \ = \  \  O\((\l_{\triv}\log \l_{\triv})^{1/2+1/B}\)\,.
\end{equation}
The implied constants in (\ref{expthmtriv}) and
(\ref{expthmnontriv}) are absolute, while the one in
(\ref{expthmbd}) depends only on $B$ and $n$.
\end{thm}

\begin{rem}\label{minkow}
a) The Theorem immediately applies to quotients of narrow ray class groups, such as ray class groups themselves.  This is because the spectrum of the quotient Cayley graph consists of  eigenvalues for those characters which factor through the quotient.

b)
The parameter $q$ should be thought of as large, in light of
Minkowski's theorem that there are only a finite number of number
fields with a given discriminant \cite[p.~121]{lang}.

c) The above
bound on the spectral gap is worse than that for Ramanujan graphs \cite{LPS}, and
thus abelian graphs are not optimal in this sense (see also \cite{alonreich,lubweiss}).  However, one gains explicit constructions that are simpler
computationally; additionally, there are situations where these graphs
occur naturally and the expansion bounds are helpful, as
our following examples show.
\end{rem}

From the abovementioned relationship between expander graphs and rapid
mixing of random walks, we obtain the following application.

\begin{cor}\label{mixingcor}  Fix $B>2$ and $n\ge 1$, and assume the
  same hypotheses of the previous theorem, including the choice of
  $x\ge(\log q)^B$.  Then there exists a positive constant $C$ with the following property:~for $q$ sufficiently large, a random walk of length $$t \ \ \ge \ \
  C\ \f{\log|G|}{\log\log q}$$ from any starting vertex lands in any fixed subset $S \subset G$
  with probability at least $\f{1}{2}\f{|S|}{|G|}$.
\end{cor}

Let us now illustrate the theorem with a few examples.  The first
example is the field $K=\Q$, whose  narrow ray class groups
are of the form $(\Z/q\Z)^*$, for $q>1$. In this case the edges of the
Cayley graph
connect each vertex $v \in (\Z/q\Z)^*$ to $pv$ and $p\i v$ (mod
$q$), for all primes $p$ such that $p \le (\log q)^{2+\d}$ and $p\nmid
q$. Starting from any $v$ and taking random steps of this form
results in a uniformly distributed random element of $(\Z/q\Z)^*$ in
$O((\log q)/\log \log q)$ steps. The character sum
(\ref{abelcayeign}) for $\l_\chi$ here amounts to the sum $2\, \Re
\sum_{p\le(\log q)^{2+\d}}\chi(p)$, so bounds on $\l_\chi$ yield
statements about the distribution of small primes in residue classes
modulo $q$.   GRH, which is used in (\ref{expthmbd}), is a natural
tool for such problems.  It seems difficult to obtain an
unconditional result along these lines, because the special case
when $\chi$ is a quadratic character modulo $q$ is related to the
problem of estimating the smallest prime quadratic nonresidue modulo
$q$. Finding such a prime is equivalent to obtaining any cancellation
at all in the sum $\sum_{p\le x}(\f{p}{q})$, and even this problem
seems to require a strong hypothesis such as GRH.  However, it is possible to use the Large Sieve to prove unconditional results for typical values of $q$, such as \cite[Theorem 3]{garaev}, which shows that $\f{\l_\chi}{\l_{\triv}}$ goes to zero outside of a sparse subset of moduli $q$.

The next example, when $K$ is an imaginary quadratic number field,
is related to  elliptic curves  over finite fields.  Using the
correspondence between ordinary elliptic curves and ideal classes in
orders of imaginary quadratic number fields, we prove the following theorem.

\begin{defn}
We say that two ordinary elliptic curves $E_1, E_2$ defined over
$\F_q$ have the \emph{same level} if their rings of
 endomorphisms
$\operatorname{End}(E_i)$ are isomorphic.  (In this paper, we follow the
standard convention that $\End(E)$ refers to  $\bar{\F}_q$-endomorphisms.)
\end{defn}

\begin{thm}\label{isographexpands}
  Consider the set $S_{N,q}$ of $\bar{\F}_q$-isomorphism classes of
  ordinary elliptic curves defined over $\F_q$ having $N$
  points. Fix\footnote{\label{footnote2}We will frequently treat the elements of
  $S_{N,q}$ as curves,  though strictly speaking they are
  isomorphism classes of curves.  This distinction does not affect \thmref{ranreduce} because isomorphisms between curves in $S_{N,q}$ can be computed in time polynomial in $\log q$.}
  an $E \in S_{N,q}$ and let $\mathcal V$ be the set of all curves in $S_{N,q}$
  having the same level as $E$. Form a graph on the set of vertices
  $\mathcal V$ by connecting curves $E_1$ and $E_2$ with an edge if there
  exists an isogeny of prime degree less than $(\log 4q)^{B}$
  between them, for some fixed $B>2$. Then, assuming GRH, this graph
  is an expander graph in the sense that its nontrivial eigenvalues
  satisfy the bound (\ref{expthmbd}).
\end{thm}

Theorem~\ref{isographexpands} has implications for the security of the
elliptic curve discrete logarithm problem.  Recall that the discrete
logarithm problem (\textsc{dlog}) asks to recover the exponent $a$ of
a power $g^a$ of a known element $g$.  Its presumed difficulty serves
as the basis of several cryptosystems, for example the Diffie-Hellman
key exchange.  Though many difficult problems in computer science are
only hard in rare instances, good cryptosystems typically must be
based on problems which are almost always hard. We recall that the
\textsc{dlog} problem on a given group has \emph{ random
self-reducibility}: that means given an algorithm ${\cal A}(g^a)=a$
which solves \textsc{dlog} on, say, half of all input values $y$, we
may easily find a random value of $r$ such that $\cal A$ works on
$y'=g^r y$, and deduce that ${\cal A}(y)={\cal A}(y')-r$. Therefore,
if \textsc{dlog} is hard for some values of $y$, it must be hard for
almost all values.  Though this result says nothing about the absolute
difficulty of the problem, it is a comforting assurance regarding the
relative difficulty of multiple instances of the problem.

Elliptic curve cryptography~\cite{koblitz,miller,blake} is based on
the conjectured difficulty of \textsc{dlog} problems within the group
of points of an elliptic curve over a finite field.  At present,
cryptographers typically select elliptic curves in the following way:
a large finite field $\F_q$ is selected, and an elliptic curve
$E/\F_q$ is generated at random.  Its order $\#E(\F_q)$ is quickly
computed \cite{schoof,sea}, and the curve is discarded unless the
order has a large prime factor (because otherwise \textsc{dlog} is
much easier).  It is also checked from the point count whether or not
$E$ is supersingular or has other weaknesses, and if it is then the curve is
discarded.\footnote{Supersingular curves are thought to be
  cryptographically weaker, because of the existence of subexponential
  attacks on their \textsc{dlog} problems~\cite{mov}.  This is not to
  say that no subexponential attacks exist for ordinary curves; in
  fact, some are known to succeed on a very modest proportion of them
  \cite{ghs,traceattack}, and of course other unknown ones may yet be
  discovered.  The supersingular analog of
  Theorems~\ref{isographexpands} and~\ref{ranreduce} are given in
  \cite[Appendix]{jmv}. }  The above practice efficiently yields
elliptic curves thought to be suitable for cryptographic purposes.  An obvious
question is whether or not other considerations are important,
i.e.~whether the point count is the only factor influencing the
difficulty of \textsc{dlog} on an elliptic curve over a fixed finite
field.

In studying this question, the random self-reducibility fact from
above does not apply, because it pertains only to a single curve, and
says nothing about the comparative difficulty of \textsc{dlog} between
two different curves. However, we can instead use the fact that an
efficiently computable isogeny provides a reduction of the
\textsc{dlog} problems between two curves.  Furthermore, a theorem of
Tate~\cite{tate1} states that all curves of cardinality $N$ defined
over $\F_q$ are isogenous, but unfortunately not all isogenies are
efficiently computable, so the theorem does not immediately imply that
all curves in $S_{N,q}$ have equivalent \textsc{dlog} problems.  On
the other hand, isogenies of low degree are efficiently computable,
and the rapid mixing in Theorem~\ref{isographexpands} says that their
random compositions become uniformly distributed over curves within
each level in $S_{N,q}$.  This property allows us to establish that
the difficulty of the elliptic curve \textsc{dlog} problem is in a
sense uniform over any given level.  More precisely:

\begin{thm}\label{ranreduce}
With the hypotheses of Theorem~\ref{isographexpands}, assume there
is an algorithm $\cal A$ which solves the discrete logarithm problem
on a positive fraction $\mu$ of the elliptic curves in a given
level. There exists an absolute polynomial $p(x)$ such that one can
probabilistically solve the discrete logarithm problem on any curve
in the same level with expected runtime $\f{1}{\mu}p(\log q)$ times the maximal runtime of $\cal A$.
\end{thm}

In practice, the level restriction in Theorem~\ref{ranreduce} is
actually irrelevant. Indeed, if two curves in $S_{N,q}$ are not of the
same level, then their levels must differ at either a small prime or a
large prime.\footnote{There is possibly an intermediate range, though its existence is somewhat fluid depending on hardware and software developments (see \secref{subsec:gaps}).} In the small prime case, we can still obtain
\textsc{dlog} reductions using low degree isogenies
(cf. \secref{sec:levels}), and in the large prime case, no constructible
examples of such pairs of curves are known. Several interesting
theoretical questions remain concerning the large prime case and the
true value of the isogeny degrees needed to achieve expansion. We
describe some open problems in \secref{problemsec}.

\section{Expander Graphs}\label{expansec}

In this section we recall a standard bound for the mixing time of a
random walk on an expander graph, discuss the lack of nontrivial
short cycles on the GRH graphs, and prove  \thmref{expthm} and
Corollary~\ref{mixingcor}.  We keep the notation and definitions of
the introduction.

\begin{lem}\label{mixingtime}
Let $\G$ be a finite $k$-regular graph for which the nontrivial
eigenvalues $\l$  of the adjacency matrix $A$ are bounded by $|\l|
\le c$, for some $c<k$. Let $S$ be any subset of the vertices of
$\G$, and $v$ be any vertex in $\G$. Then a random walk of any
length at least $\f{\log 2|\G|/|S|^{1/2}}{ \log k/c}$ starting from
$v$ will end in $S$ with probability between  $ \f{1}{2}\f{|S|}{
|\G|}$ and $\f{3}{2}\f{|S|}{ |\G|}$.
\end{lem}

\noindent Of course the probability range can be significantly
narrowed by lengthening the walk, as it turns out even by a slight
amount.

\begin{proof}
Letting $\chi_S$ and $\chi_{\{v\}}$ denote the characteristic
functions of the sets $S$ and $\{v\}$, respectively, the number of
paths of length $t$ which start at $v$ and end in $S$ is given by
the $L^2$-inner product $\langle \chi_S,A^t \chi_{\{v\}}\rangle$.
Let $P$ denote the projection from $L^2(\G)$ onto the orthogonal
complement of the constant functions; the operator $A$ preserves  this space and its operator
norm on it is bounded by $c$ because of our eigenvalue assumption.  Then
\begin{equation}\label{innerproductcalc}
\langle \chi_S,A^t \chi_{\{v\}}\rangle \ \ = \ \ \f{|S|}{|\G|}k^t \
+ \ \langle   P\chi_S, A^tP\chi_{\{v\}}    \rangle\,.
\end{equation}
The latter term is bounded by
\begin{equation}
\aligned
\label{secondinnerproductbd}
    \left|  \langle   P\chi_S, A^tP\chi_{\{v\}}    \rangle \right|
     \ \ \le \ \ \|P\chi_S\|\,\|A^tP\chi_{\{v\}}\| \ \ \le
     \ \ c^t \|P\chi_S\|\,\|P\chi_{\{v\}}\| \\
      \ \ \le \ \ c^t \|\chi_S\|\,\|\chi_{\{v\}}\|
      \ \ = \ \ c^t\,|S|^{1/2}\,.
\endaligned
\end{equation}
For $t\ge \f{\log 2|\G|/|S|^{1/2}}{ \log k/c}$ this is at most half
the size of the main term $k^t|S|/|\G|$ from
(\ref{innerproductcalc}), as was to be shown.
\end{proof}

Next we come to the topic of \emph{girth}, the length of the shortest
 closed cycle on the graph.  Graphs with large girth are
important in many applications, for example to the design of
collision resistant hash functions and stream ciphers (see, for
example, \cite{goldreich,horwen,mv3}). The girth of a $k$-regular
graph cannot be larger than $2\log_{k-1}|\G|$. This inequality comes
from counting the number of points $b(r)$  in a ball of radius $r$
in a $k$-regular tree; a graph with girth $\g$ satisfies the
inequality $b(\g)\le |\G|$, which gives an upper bound on $\g$.
Random graphs tend to have small girth, but one can use
probabilistic methods to show the existence of graphs having girth
at least $(1+o(1))\log_{k-1}|\G|$, i.e.~roughly half the optimal
size. The LPS Ramanujan graphs have the largest known girths:
$(4/3+o(1))\log_{k-1}|\G|$ \cite{LPS,sharplps}.  It is an open
question as to how large the girth can be.

 Abelian Cayley
graphs cannot have large girth because they have many short cycles
of the form $xyx\i y\i$.  To rule out these, one can speak of the
\emph{nonabelian girth}, which is the shortest cycle not having steps
both of the form $x^a$ and $x^{-b}$ for $a,b>0$ and $x\in S$. We
remark that the graphs on $(\Z/q\Z)^*$ described just after
Theorem~\ref{expthm} have
\begin{equation}\label{nonabelgirth}
    \text{nonabelian girth of~}\G \ \ge \ (1+o(1))\log_{k-1}|\G|\,.
\end{equation}
 Indeed, a cycle amounts to two products of
small primes which are equal modulo $q$; by unique factorization, at
least one of these products must be larger than $q$, which gives a
lower bound on the number of factors.  This argument also gives the
same lower bound for the \emph{odd girth} of $\G$ (i.e.~the shortest
closed cycle of odd length), which again is relatively large.  It
should be noted, however, that this is not optimal; in
fact
 there are code-based constructions \cite{alontough,alonreich} which have nonabelian girth
at least $(2+o(1))\log_{k-1}|\G|$.   The reason for mentioning this,
though, is that  explicit examples of graphs with large nonabelian
girth are important for cryptographic applications.

We conclude this section with the proofs of \thmref{expthm} and
Corollary~\ref{mixingcor}.

\begin{proof}[Proof of \thmref{expthm}]
 We explained in (\ref{abelcayeign}) and in the
remarks following it that (\ref{expthmtriv}) and
(\ref{expthmnontriv}) follow from the following estimates for sums
of characters $\chi$ of $G$:
\begin{equation}\label{needtoshow}
\aligned
  \sum_{N{\frak p}\,\le \,x\text{~prime}}\(\chi({\frak p})
  +\chi({\frak p})\i\) \ \ & =
   \ \  2\,\Re\,\sum_{N{\frak p}\,\le \,x\text{~prime}}\chi({\frak p}) \ \ \\
   & = \ \
     2\,r\,\operatorname{li}(x) \ + \   O\(n\,x^{1/2}\log(xq)\)\,,
     \endaligned
\end{equation}
with an absolute implied constant.  Here $r=1$ if $\chi$ is the
trivial character, and 0 otherwise.  Of course,  $\chi$ may be
viewed as a Hecke Grossencharacter on $I_{\mathfrak{m}}$ which is trivial
on $P_{\mathfrak{m}}^+$. Hecke proved that its $L$-function
\begin{equation}\label{lschi}
    L(s,\chi) \ \ = \ \ \sum_{{\frak a}\text{ integral ideal}} \chi({\frak a})\,(N{\frak
    a})^{-s} \ \ = \ \ \prod_{{\frak p}\text{ prime ideal}} (1\,-\,\chi({\frak p})\,(N{\frak
    p})^{-s})^{-1}
\end{equation}
 analytically continues to a holomorphic function on $\C-\{1\}$ of order 1, with at most a simple pole at $s=1$ which occurs only
when $\chi$ is the trivial character.  Furthermore, he also
established a standard functional equation for its
completed $L$-function, which is a product of $L(s,\chi)$,
$\G$-factors of the form $\G(\f s2)$, $\G(\f{s+1}{2})$, and $\G(s)$,
and a power $Q^{s/2}$ of some integer $Q>0$
\cite[p.~211]{iwaniec-blue}.
 The value
of $Q$ varies with different characters, but is always bounded above
by $q=D\cdot N\mathfrak{m}$.  The Dirichlet series coefficients of $L(s,\chi)$, like those of any Artin $L$-function, satisfy the Ramanujan-Petersson conjecture.

Using these analytic properties, along with the assumption of GRH,
one can derive the following standard estimate (which is found in \cite[p.~114]{iwaniec}):
\begin{equation}\label{iwanshows}
\sum_{N{\frak p}\,\le \,x\text{~prime}} \chi({\frak p})\log(N{\frak
p}) \ \ = \ \
     r\,x \ + \   O\(n\,x^{1/2}\log(x)\log(xQ)\)
\end{equation}
for primitive characters $\chi$, again with an absolute implied
constant.  If $\chi$ is imprimitive, one must also include terms for
prime ideals $\frak p$ dividing $\mathfrak{m}$.  There are at most
$O(\log N{\mathfrak{m}})=O(\log q)$ of these, so both their contribution
and the existing error term in (\ref{iwanshows}) can be safely
absorbed into the enlargened error term
$O(nx^{1/2}\log(x)\log(xq))$.
  This variant of (\ref{iwanshows}) in turn implies (\ref{needtoshow}) by
a simple application of partial summation.
\end{proof}

\begin{proof}[Proof of Corollary~\ref{mixingcor}]
The proof follows from \lemref{mixingtime} once we have verified
that $\log \f{k}{c}$ is bounded below by a  constant (depending on $B$
and $n$) times $\log\log q$ once $q$ is sufficiently large. Indeed, in our setting
 the degree is
$k=\l_{\triv}$, and $c$ may be taken to be the bound in
(\ref{expthmbd}).  For $q$ sufficiently large, $\log \f{k}{c}$ is indeed
bounded below by a constant times $\log \l_{\triv} \gg B \log\log q$.
\end{proof}

\begin{rem}
The main point of the Corollary is to give examples of rapid mixing
over large graphs.  However, for a finite number of cases when $q$ is small, the graph $\G_x$ may
actually be disconnected.  In addition, the equidistribution is not as interesting in situations when the graph $\G_x$ has  relatively few vertices, i.e.~when the narrow ray class number
of $\mathfrak{m}$ is small. This can be computed explicitly as
\begin{equation}\label{rayclassnumber}
    |G| \ \ = \ \ |\G_x| \ \ = \ \ 2^{r_1}\,\f{h(K)\,|({\cal O}_K/{\frak
    m})^*|}{[U(K):U_{\mathfrak{m}}(K)]}\, ,
\end{equation}
where $r_1$ is the number of real embeddings of $K$, $h(K)$ its class number, ${\cal O}_K$ its ring of
integers, $U(K)$ its unit group, and $U_{\mathfrak{m}}(K)\subset U(K)$ its subgroup
of totally positive units which are congruent to 1 $(mod~ {\frak
m})$ \cite[Prop. 3.2.4]{cohen}.  For a fixed degree $n$, the class number $h(K)$ is $O_\e(|D|^{1/2+\e})$ for any $\e>0$, and so $|G|$ above is bounded by $O(q)$.
\end{rem}

\section{Elliptic curves}\label{sec:ec}

In this section we explain the connection between the GRH graphs and
elliptic curves, and prove Theorem~\ref{isographexpands}. For ease of presentation, we begin first with the
case of elliptic curves defined over complex numbers, and then later
explain how our results over complex numbers imply the corresponding
results over finite fields.

Let $\O_D$ be an imaginary quadratic order of discriminant $D <
0$. Denote by $\Ell(\O_D)$ the set of all isomorphism classes of
elliptic curves $E$ over $\C$ having $\O_D$ as their full ring of
complex multiplication (i.e. having $\End(E) \iso \O_D$). It is well
known that isomorphism classes of elliptic curves over $\C$ correspond
bijectively with homothety classes of complex
lattices~\cite[I.1]{silverman2}; accordingly, we will write
$E_\Lambda$ throughout for the elliptic curve corresponding to a
complex lattice $\Lambda \subset \C$. Moreover, fixing an embedding
$\O_D \subset \C$, one can show that ideal classes $\mathfrak{a}
\subset \O_D$ give rise to precisely those
lattices representing elliptic curves in
$\Ell(\O_D)$~\cite[10.20]{cox}, and that the map $\mathfrak{a} \mapsto
E_\mathfrak{a}$ induces a bijection between the ideal class group
$\Cl(\O_D)$ of $\O_D$ and $\Ell(\O_D)$.

The above paragraph thus explains the correspondence between ideal
class groups and elliptic curves over $\C$. The following proposition
describes how this correspondence behaves with respect to isogenies:

\begin{prop}\label{isogeny-correspondence}
~
\newline

\vspace{-.7cm}

\begin{enumerate}
\item\label{transitive-action} There is a well defined simply
  transitive action of $\Cl(\O_D)$ on
$\Ell(\O_D)$, given by the formula
$$
\mathfrak{a} * E_\Lambda := E_{\mathfrak{a}^{-1} \Lambda},
$$
valid for any nonzero fractional ideal $\mathfrak{a} \subset
\O_D$.
\item\label{isogeny-existence} If $\mathfrak{a}$ is an invertible
  ideal of $\O_D$, one has $\Lambda \subset \mathfrak{a}^{-1}\Lambda$, and this inclusion induces an
  isogeny $E_\Lambda \to \mathfrak{a} * E_\Lambda$ of degree
  equal to the norm $N(\mathfrak{a})$ of the ideal $\mathfrak{a}$.
\item\label{isogeny-uniqueness} Up to isomorphism, every isogeny
  between two elliptic curves $E_1, E_2 \in \Ell(\O_D)$ arises in
  the above manner.
\end{enumerate}
\end{prop}
\begin{proof}
Items~\ref{transitive-action} and~\ref{isogeny-existence} are proved
in~\cite[II.1]{silverman2} (for the case of $\O_D$ maximal)
and~\cite{langellipticfunctions} (for the general case).

To prove item~\ref{isogeny-uniqueness}, let $\phi\colon E_1 \to E_2$
be an isogeny and choose fractional ideals $\mathfrak{a} \subset
\mathfrak{b}$ of $\O_D$ such that $E_\mathfrak{a}\iso E_1$ and
$E_\mathfrak{b} \iso E_1/\ker(\phi) \iso E_2$. Since $\mathfrak{a}
\subset \mathfrak{b}$, there exists an integral ideal $\mathfrak{c}
\subset \O_D$ such that $\mathfrak{b} \mathfrak{c} = \mathfrak{a}$,
whereupon the morphism $\psi\colon E_\mathfrak{a} \to \mathfrak{c} *
E_\mathfrak{a}$ yields an isogeny which has the same kernel as $\phi$,
and hence must be isomorphic to $\phi$.
\end{proof}

We now state and prove an analogue of Theorem~\ref{isographexpands}
over the complex numbers.

\begin{thm}\label{complex-expanders}
Let $\Gamma$ be the graph whose vertices are elements of $\Ell(\O_D)$
and whose edges are isogenies of prime degree less than some fixed
bound $M \geq (\log|D|)^{B}$, for some absolute constant $B> 2$. Then, assuming
GRH, the graph $\Gamma$ is an expander graph satisfying the bound (\ref{expthmbd}).
\end{thm}

\begin{proof}
We have already seen that the elements of $\Ell(\O_D)$ are in
bijection with the elements of the group
$\Cl(\O_D)$~\cite[10.20]{cox}, and that the action of $\Cl(\O_D)$ on
$\Ell(\O_D)$ defined in Proposition~\ref{isogeny-correspondence}
coincides exactly with the translation action of $\Cl(\O_D)$ on
itself under this bijection. Moreover, isogenies of prime degree less
than $M$ correspond to integral ideals of prime
norm less than $M$, and the inverses (i.e. complex
conjugates) of these ideals have the same prime norm and thus also
yield such isogenies. Hence, the graph $\Gamma$ is
isomorphic to the Cayley graph of $\Cl(\O_D)$ under the generating set
consisting of ideals of prime norm less than $M \geq (\log|D|)^{B}$.

Next we relate this graph to one covered by Theorem~\ref{expthm}.  Let $K=\Q(\sqrt{D})$ and $\frak m$ the principal ideal generated by the conductor $c$ of the discriminant $D$ (i.e.~the largest integer whose square divides $D$).  Then the class group $\Cl(\O_D)$ is a quotient of the narrow ray class group of $K$ relative to $\frak m$ \cite[Prop. 7.22]{cox}, and Theorem~\ref{expthm} applies directly to $\Gamma$ and
equation~\eqref{expthmbd} with $x = M$ gives the desired bound.
\end{proof}

In order to prove Theorem~\ref{isographexpands} from
Theorem~\ref{complex-expanders}, we require the following classical
result, known as Deuring's lifting theorem~\cite{deuring}:

\begin{thm}\label{deuring}
~
\newline

\vspace{-.7cm}
\begin{enumerate}
\item  Let $E$ be an elliptic curve defined over $\F_q$, and let $\phi$ be
  a nontrivial endomorphism of $E$. There exists an elliptic curve
  $\tilde{E}$ defined over a number field $L$, a prime ideal
  $\mathfrak{p}$ of $L$, and an endomorphism $\tilde{\phi}$ of
  $\tilde{E}$ such that $\tilde{E}$ and $\tilde{\phi}$ reduce to $E$
  and $\phi$ modulo $\mathfrak{p}$.
\item When $E$ is ordinary, the mod $\mathfrak{p}$ reduction map induces
an isomorphism $\End(\tilde{E}) \iso \End(E)$.
\end{enumerate}
\end{thm}

\begin{proof}[Proof of Theorem~\ref{isographexpands}:] Since the curves in
Theorem~\ref{isographexpands} are ordinary, there exists an imaginary
quadratic order $\O_D$ such that $\End(E) = \O_D$. Observe
that $(\log 4q)^{B} \geq (\log |D|)^{B}$, since $D = t^2 - 4q$ where the trace $t$ satisfies the
Hasse bound $|t| < 2\sqrt{q}$. Hence $(\log 4q)^{B}$ satisfies
the condition for $M$ in Theorem~\ref{complex-expanders}.

We will now show that the graph $\Gamma$ in
Theorem~\ref{complex-expanders} is isomorphic to the graph defined in
Theorem~\ref{isographexpands}. The elliptic curves in $\Ell(\O_D)$ are
all defined over the ring class field $H$ of $\O_D$. Identification of
the vertices is accomplished by choosing a prime $\mathfrak{p} \subset
H$ lying over the characteristic $p$ of $\F_q$, and reducing curves in
$\Ell(\O_D)$ to obtain curves in $S_{N,q}$. Theorem~\ref{deuring} shows that this identification is surjective.
To show that it is injective, consider two non-isomorphic curves
$E_\mathfrak{a}$ and $E_\mathfrak{b}$ in $\Ell(\O_D)$, meaning that
$\mathfrak{a}$ and $\mathfrak{b}$ lie in different ideal classes in
$\Cl(\O_D)$.  By the Chebotarev density theorem, there exists
an unramified prime ideal $\mathfrak{c}$
belonging to the same ideal class as $\mathfrak{a} \mathfrak{b}^{-1}$; note
in particular that $\mathfrak{c}$ is not principal.
By Proposition~\ref{isogeny-correspondence}, the ideal $\mathfrak{c}$
induces an isogeny
$\phi$ between $E_\mathfrak{a}$ and $E_\mathfrak{b}$ having degree
equal to $N(\mathfrak{c})$. If the reductions
$\bar{E}_\mathfrak{a}$ and $\bar{E}_\mathfrak{b}$ of
${E}_\mathfrak{a}$ and ${E}_\mathfrak{b}$ modulo $\mathfrak{p}$ were
to be somehow isomorphic, then $\phi$ would represent an endomorphism
of $\bar{E}_\mathfrak{a}$, of degree $N(\mathfrak{c})$.   However, we know the endomorphism ring of
$\bar{E}_\mathfrak{a}$ is
equal to $\O_D$, and no element of $\O_D$ has norm equal to
$N(\mathfrak{c})$ (this is because $\Q(\sqrt{D})$ is an imaginary quadratic number field).
Thus the
endomorphism ring $\O_D$ cannot contain any endomorphism of degree equal
to $N(\mathfrak{c})$.

Likewise, for each prime $\ell < (\log 4q)^{B}$, the reduction
map modulo $\mathfrak{p}$ sends every isogeny of degree $\ell$ in
characteristic $0$ to an isogeny of degree $\ell$ in characteristic
$p$. All isogenies in characteristic $p$ are obtained in this way,
since isogenies of degree $\ell$ are given by the roots of the modular
polynomial $\Phi_\ell(x,y)$, and this polynomial does not have more
roots over the algebraic closure in characteristic $p$ than in characteristic $0$.
\end{proof}

\section{Relationship with discrete logarithms}\label{dlog}

Given a generator $g$ of a cyclic group $G$ of order $n$, the discrete
logarithm of an element $h$ of $G$ is defined to be the residue class $x$ of
integers mod $n$ such that $g^x = h.$ The elliptic curve discrete logarithm
problem is the problem of computing discrete logarithms when $G$ is the
group of points on an elliptic curve defined over a finite field $\F_q$.
Determining the difficulty of this problem is important because much of
elliptic curve cryptography is based, at least conjecturally, on the
infeasibility of computing discrete logarithms on elliptic curves
defined over a finite field.

   Galbraith~\cite{gal99} has observed that given an
efficiently computable isogeny $\phi\colon E \to E'$, one can
compute discrete logarithms on $E$ by computing discrete logarithms
on $E'$. The procedure is as follows:~given $P,Q \in E$, compute
$\phi(P)$ and $\phi(Q)$, and determine the discrete logarithm $x$ of
$\phi(Q)$ on $E'$ with respect to the generator $\phi(P)$. The equation
$x \cdot \phi(P) = \phi(Q)$ determines the solution for $x$ modulo the kernel of $\phi$.  When $\phi$ is furthermore a low-degree isogeny, it is both efficiently computable and has small kernel (which itself can be efficiently enumerated).  Such an isogeny provides a
reduction between the discrete logarithm problems on $E$ and
$E'$, in time polynomial in $\log q$ and the degree. Moreover, a theorem of Tate~\cite{tate1} states that two
elliptic curves $E$ and $E'$ defined over $\F_q$ have the same number
of points if and only if they are isogenous.  Tate's theorem guarantees the existence of an isogeny defined over $\F_q$ between curves in their equivalence classes, which computationally amounts to one between the curves themselves (see footnote~\ref{footnote2}).  However, this isogeny usually is  difficult to compute and has enormous degree.

We now use the above observation to give a proof of
Theorem~\ref{ranreduce}. Our proof consists of showing that, for
curves of the same level, a composition of low-degree isogenies between them exists.  Indeed, though the degree of such a composition may be very large, it can be computed efficiently; furthermore, it gives efficient reductions between all curves it connects.

\begin{proof}[Proof of Theorem~\ref{ranreduce}:]
Returning to the isogeny graph of \thmref{isographexpands}, let $S$ denote the
  subset consisting of the $\mu$-fraction of elliptic curves to which
  the algorithm $\A$ applies. Let $E$ be any curve of the same level
  as the curves in $S$.  Because of the effective upper bounds on class numbers, one has that $\log|S_{N,q}| \le c' \log q$, for some $c'>0$.
  Construct  a random walk of length
  $Cc' (\log q)/\log \log q$ starting at $E$, where $C$ is
  the constant in Corollary~\ref{mixingcor}.  Let $\phi$ denote the isogeny
  equal to the composition of the isogenies represented by the edges
  comprising the random walk. Then $\phi$ can be evaluated in
  polynomial time, and hence the discrete logarithm problem on $E$ can
  be solved efficiently by querying $\A$, as long as the random walk
  above lands in $S$.  By Corollary~\ref{mixingcor}, the probability
  that the random walk lands in $S$ is at least $\frac{\mu}{2}$, so by
  repeating this process until the walk lands in $S$, we can solve
  discrete logarithms on $E$ in probabilistic polynomial time using an
  expected number of queries to $\A$ bounded by $\frac{2}{\mu}$.
\end{proof}

\section{Reductions between different levels}\label{sec:levels}

It is natural to ask whether the equivalence of discrete logarithms
holds for elliptic curves in different levels. We begin by observing
that the CM field $\End(E) \otimes \Q$ is the same for all curves $E
\in S_{N,q}$ regardless of level. Moreover, two curves $E, E'$ have
the same level if and only if the conductors of their endomorphism
rings in $\End(E) \otimes \Q$ are equal. It is thus natural to define
the \emph{conductor gap} to be the value of the largest prime factor
at which the prime factorizations of the conductors of $\End(E)$ and
$\End(E')$ differ; in addition, for a single curve $E$ we define the
conductor gap of $E$ to be the maximal possible conductor gap over
all possible pairs of isogenous $E, E'$. The conductor gap provides
a rough measurement of how much the levels of $E$ and $E'$ differ.

Given any curve $E$ whose endomorphism ring has conductor $c$, it is
possible to compute a curve $E'$ with conductor $c\ell$ together with
an isogeny $E \to E'$ of degree $\ell$ in time $O(\ell^3)$; the reverse, starting from $E'$ of
conductor $c\ell$ and ending up with $E$ of conductor $c$, is also
possible in the same amount of time (\cite{koh96, gal99, engemodpoly}).
 Consider a union of any number of levels which collectively have conductor gap  bounded polynomially in $\log q$.  Though the individual sizes of each level may be difficult to compute, formula (\ref{rayclassnumber}) or \cite[Cor.~7.28]{cox} allows one to compute their relative sizes efficiently.  By weighing these sizes it is possible to select a level at random with probability proportional to its total size amongst this union.  This level can be reached by appropriate low degree isogenies.  Thus it is possible to reach a random curve through walks of low degree isogenies, and
 it follows that
Theorem~\ref{ranreduce} holds for the union of any number of levels
which collectively have conductor gap  bounded polynomially in $\log q$.

Large conductor gaps do pose an obstacle in the statement of
Theorem~\ref{ranreduce}, but they rarely arise in practice. Indeed,
every curve $E \in S_{N,q}$ has at least the endomorphisms $\Z \subset
\End(E)$ and $\pi_q \in \End(E)$, with $\pi_q$ denoting the Frobenius
endomorphism. The discriminant of the quadratic order $\Z[\pi_q]$ is
equal to $t^2 - 4q$ where $t = q+1-N$, and the conductor of any curve
in $S_{N,q}$ must be an integer $c$ satisfying $c^2 \div (t^2 -
4q)$. Thus, if $t^2-4q$ is square free, then all curves in $S_{N,q}$
are of the same level, and in this case the level restriction in
Theorem~\ref{ranreduce} is vacuous. More generally, as long as
$t^2-4q$ has no large repeated prime factors, the statement of
Theorem~\ref{ranreduce} holds for all of $S_{N,q}$, by the previous
paragraph.

We can analyze the expected frequency of large conductor gaps as
follows. The Hasse-Weil bound on $t$ implies $-4q \leq t^2 - 4q \leq
0$. A random integer within this interval has probability $1 -
\prod_{p>\beta}^{2 \sqrt{q}} (1-p^{-2})$ of admitting a repeated prime
factor $p > \beta$. Since this probability is bounded above by
$O(1/\beta)$, we expect as a heuristic that, for any positive $\beta<p$,
random choices of $(N,q)$ will admit repeated prime factors exceeding
$\beta$ with probability $1/\beta$.  In fact, \cite[Theorem 1]{lucashpar}
rigorously proves the probability estimate $\f{(\log p)^2}{\beta}$ for $\beta \ll p^{1/6}$, where $q=p$ is odd.
Therefore, in most cases,
conductor gaps between elliptic curves are quite small and we can
ignore the effects of differing endomorphism rings in our discrete
logarithm comparisons.  For example, an investigation of nine randomly
generated curves listed in international standards documents reveals
that all of them satisfy $c_{N,q} \leq 3$
(cf. \secref{nistsection}). In fact, a somewhat surprising observation
is that there is currently no efficient algorithm to construct pairs
of elliptic curves with conductor gaps that are \emph{not} small, even
though such pairs are known to exist in abundance
(cf. \secref{problemsec}).

\section{Government standards for curves}\label{nistsection}

In the previous section we showed that all curves in an isogeny class
have identical security on average whenever the conductor gap is
small.  However, determining the conductor gap of a curve requires
factoring a large integer and hence is a nontrivial computation. In
this section we provide the computation of the conductor gap for a
family of randomly selected curves which appear as part of a US
government standard.

In 2000 the National Institute of Standards and Technology (NIST), a
branch of the United States Department of Commerce, introduced a
family of elliptic curves as standards for cryptographic
applications \cite{standardsdocumentsnist}.  The selection of these curves was the
outcome of several years of testing.   The NIST curves are
generated by the values of secure hash functions applied to
publicly-revealed seeds, making it plausible that they were not excessively manipulated before their public
release.  However, the user cannot be totally confident that there is not a backdoor or weakness in the
published curve.

Though it is hard to imagine arguing directly that discrete logarithms on a
\emph{specific} elliptic curve do not have good attacks, our results
can be used to give some assurance that
 the NIST curves are not weaker than comparable
elliptic curves.
Namely,
\thmref{ranreduce} and the comment immediately following it
 show that the discrete logarithm problem has roughly equivalent difficulty as one ranges over curves defined over the same field, and whose  endomorphism rings have small
conductor gap.

Some of the NIST curves are \emph{Koblitz}
curves~\cite{cmkoblitz}, which are not expected to have small
conductor gaps.  However, for the remaining  NIST curves, some
lengthy computations showed that the conductor gap is very small:~all
but one curve had a conductor gap of 1, and the only exception had a
conductor gap of 3.  That means that in the former cases, the isogeny class consists of only one level, and \thmref{ranreduce} provides a full equivalence of discrete logarithms.  Only in the exceptional case with conductor gap 3 must one navigate between levels (the topic of \secref{sec:levels}); this can easily be done  by constructing a degree 3 isogeny between them.
Therefore we may conclude that these curves have
typical difficulty among all elliptic curves defined over the same
field and having the same number of points.

As an  example, consider the NIST curve B-571, which is
given by the Weierstrass equation $ y^2 + xy = x^3 + x^2 + b$ over
the field $\F_{2^{571}}$.  Here $b$ is an element of
$\F_{2^{571}}$ which is cumbersome to describe but can be found on
p. 47 of \cite{standardsdocumentsnist}.  It has discriminant
\begin{equation}
\aligned d = -
&210092063841005638410400838462812964562253124135523060955333\backslash
\\
 & 767330638498791801056156659734237518468659692798673383993911\backslash \\
 & 78057790576859207002963481895511008772786625592941143
\endaligned
\end{equation}
and prime factorization
\begin{equation*}
\aligned
= & -137 * 1502689 * 5608493523058319 * 3563521804312876303
\\
& \ * 46393104672338327566438581332776443577 \\
& \ * 1100628851017477373738489717699925956411395060089467152067605\backslash \\ & \ \ \ \  28637300688225399301632484625559
\endaligned
\end{equation*}
(we have written out the decimal expansion of $d$ over several
lines owing to its length). One can determine the conductor gap
knowing this factorization:~it is the largest square factor, which in this example is 1.

 We wish to mention that finding
the above factorization was far from trivial, taking about 5 days
on a dedicated  cluster in the Netherlands which utilized
specialized factoring software.  Although determining the
conductor gap is useful in assuring that a given elliptic curve is
not cryptographically weak,  clearly  this is not a test which the
average user can perform.  It may be good practice for standards
bodies to publish the factorization of the discriminants along
with their recommended curves so that users have this information.

\section{Open problems} \label{problemsec}

In this section we address two shortcomings of
Theorem~\ref{ranreduce}.  The first is that the Theorem, as stated,
applies only to individual levels of curves. As noted just after its
statement and further in Sections~\ref{sec:levels} and
\ref{nistsection}, curves whose levels differ by a ratio composed of
small primes can be bridged by random isogenies; the issue is when the
conductor gap has a large prime factor.  The second is the strong
analytic assumption of the Generalized Riemann Hypothesis.  We conclude by discussing some related cryptographic problems.

\subsection{Large conductor gaps}\label{subsec:gaps}

The equivalence result of Theorem~\ref{ranreduce} is incomplete in the
sense that it does not apply to curves having a large conductor gap.
Pairs of such curves certainly exist, but no efficient method is known
for finding them, and indeed no explicit example is known at the
present time.  A curve chosen at random will have conductor greater
than $\ell$ with probability heuristically equal to $1/\ell$ (see
Section~\ref{sec:levels}).   As we mentioned in \secref{sec:levels}, it is possible to produce an explicit isogeny between two curves with conductor gap $\ell$ in time $O(\ell^3)$, which for large $\ell$ is far slower than solving discrete logarithms themselves.  Additionally,  it was recently shown in \cite{engecm} how to create special pairs of curves with conductor gap $\ell$ in time $O(\ell^2)$, without finding an explicit isogeny between them.
  All of these methods are too
slow for large values of $\ell$, but leave an intermediate range of conductor gaps which presently cannot bridged by computable isogenies.

The conductor gap question is especially pertinent for certain special
classes of curves in cryptography such as pairing friendly curves (see
\cite{taxon}).  All constructible examples of such curves are presently
restricted to small discriminants, with the exception of certain
families of curves having conductor gaps which fall within the
abovementioned intermediate range~\cite{gaetan}; note, however, that
these conductor gaps are still small enough that improvements such as
Moore's law affect the boundaries of this range. There is some concern
(although no proof) that discrete logarithms on such curves are weaker
than on general pairing friendly curves.  Achieving large conductor gaps for
pairing friendly curves would help alleviate this concern, since our
work  then implies that pairing friendly curves with large
discriminant are provably as secure as random pairing friendly curves.

\subsection{The assumption of GRH}

The theorems in this paper all assume the Generalized Riemann Hypothesis, which is used to obtain the error estimate in (\ref{iwanshows}).  Lighter analytic assumptions still imply nontrivial error estimates; for example the Generalized Lindel\"of Hypothesis instead implies a bound of $O_{\e,K}(x^{1/2+\e}Q^\e)$ for any $\e>0$  \cite{iwaniec}.  This corresponds to a subexponential time algorithm in Theorem~\ref{ranreduce}, as opposed to a polynomial time one.

An unconditional proof of expansion seems out of reach at present.  In the introduction it was explained why expansion bounds for $\l_\chi$  imply bounds on the least quadratic nonresidue, and thus at present require an analytic assumption.  The recent preprint \cite{lauwu} considers cancellation in the sums $\l_\chi$ defined in (\ref{abelcayeign})  for other characters.

Intriguingly, it has been widely speculated that the GRH implication of $B>2$ in \thmref{expthm} is  not sharp, and that $B>1$ is in fact expected. This feature dates back to the suggestion of Littlewood that the Euler product for $L(1,\chi)$ could be approximated by the partial Euler product over primes smaller than $(\log Q)^B$, for any $B>1$.  This  approximation is consistent with the best known constructions of lower bounds for the error terms in the sums (\ref{iwanshows}), and for related problems such as the least nonresidue problem  \cite{littlewood,gramrose,gransoundgafa}.  Recent work of  \cite{vaughn,montvaughn,gransoundextreme} supports the validity of the wider range $B>1$.  This bound is also sharp from the point of view of the Alon-Roichman Theorem \cite{alonreich}, which asserts that expanders must have at least logarithmic degree in the size of the graph.

Finally, the constants in (\ref{iwanshows}) are effective and numerical values for them have been obtained in \cite{bachpap1,bachpap2}.

\subsection{Generalizations to other cryptographic problems}

    The elliptic curve discrete logarithm problem can be generalized to
    Jacobians of hyperelliptic curves or other curves of higher genus, and
    recently there has been some progress in obtaining efficiently
    computable isogenies between such abelian varieties~\cite{bensmith}.
    At present, not enough such isogenies are known to enable any statement
    about reducibility of discrete logarithms between such Jacobians, but further
    developments could likely yield new results in this area.

    In a different vein, one can consider alternative cryptographic problems
    such as the Diffie-Hellman problem instead of the discrete logarithm
    problem.  For example, the recent
    paper~\cite{jv08} shows that, for curves over a prime field,
    computing the least significant bit of a Diffie-Hellman secret with
greater than 50\% probability over a non-negligible fraction of curves
    is almost always equivalent
    to solving the full Diffie-Hellman problem itself (assuming GRH).  The proof
    relies heavily on the rapid mixing properties of isogeny graphs for
    ordinary elliptic curves.

    \vspace{.6 cm}

{\bf Acknowledgements:} It is a pleasure to thank Noga Alon, R. Balasubramanian, Dan Boneh, Ehud de
Shalit, Noam Elkies, Andrew Granville, Dimitar Jetchev, Nati Linial, Alexander Lubotzky, Peter
Montgomery, Kumar Murty, Ravi Ramakrishna, Ze'ev Rudnick, Peter Sarnak, Adi Shamir, Igor Shparlinksi,
and Martin Weissman for their helpful comments and suggestions. In
particular, we wish to thank Peter Montgomery for his assistance and
advice in factoring the discriminants of NIST curves in
\secref{nistsection}, and for making computational resources available
to us.

\noindent Addresses:

\noindent David Jao\\
Department of Combinatorics and Optimization\\
University of Waterloo\\
Waterloo, ON  N2L 3G1\\
Canada\\
\texttt{djao@math.uwaterloo.ca}

\vspace{.3 cm}

\noindent Stephen D. Miller
\\  Department of Mathematics \\ 110 Frelinghuysen Road \\
Rutgers, The State University of New Jersey \\ Piscataway, NJ
08854 \\  {\tt miller@math.rutgers.edu}

 \vspace{.3 cm}

\noindent Ramarathnam Venkatesan \newline Microsoft
Research Cryptography and Anti-Piracy Group
\\ 1 Microsoft Way \\ Redmond, WA 98052 \newline
\hspace{5cm} and\\
\noindent
Cryptography, Security and Applied Mathematics Research Group, Microsoft
Research India\\  Scientia - 196/36 2nd Main,
Sadashivnagar, Bangalore 560 080, India\\
 {\tt
venkie@microsoft.com}

\end{document}